\documentclass{amsart}
\usepackage{setspace}
\usepackage{a4}
\usepackage{amssymb,amsmath,amsthm,latexsym}
\usepackage{amsfonts}
\usepackage{amsfonts}
\usepackage{graphicx}
\usepackage{textcomp}
\usepackage{cite}
\newtheorem{theorem}{Theorem}[section]

\newtheorem{corollary}[theorem] {Corollary}
\newtheorem{definition}[theorem]{Definition}
\newtheorem{example}[theorem]{Example}
\newtheorem{lemma} [theorem]{Lemma}

\newtheorem{proposition}[theorem]{Proposition}
\newtheorem{remark}[theorem]{Remark}
\newtheorem{question}[theorem]{Question}
\title{This is the title}
\usepackage{amssymb}
\usepackage{amssymb}
\usepackage{amssymb}
\usepackage{amssymb}
\usepackage{amsmath}
\usepackage{tikz}
\usepackage{hyperref}
\usepackage{enumerate}
\usepackage{mathtools}
\usepackage{amsmath}
\usepackage{tikz}
\usepackage{fancyhdr}
\begin{document}

\begin{center}
{\bf{IMPROVING CASAZZA-KALTON-CHRISTENSEN-VAN EIJNDHOVEN PERTURBATION WITH APPLICATIONS}}\\
K. MAHESH KRISHNA  \\
Department of Humanities and Basic  Sciences \\
Aditya College of Engineering and Technology \\
Surampalem, East-Godavari\\
Andhra Pradesh 533 437 India\\
Email: kmaheshak@gmail.com\\
 \today
\end{center}

\hrule
\vspace{0.5cm}

\textbf{Abstract}: Let $ \mathcal{X}$, $ \mathcal{Y}$ be  Banach spaces and $S:\mathcal{X} \to  \mathcal{Y} $ be an invertible Lipschitz map. Let     $ T : \mathcal{X}\rightarrow \mathcal{Y}$ be a map and there exist  $ \lambda_1,\lambda_2 \in \left [0, 1  \right )$ such that 
\begin{align*}
	\|Tx-Ty-(Sx-Sy)\|\leq\lambda_1\|Sx-Sy\|+\lambda_2\|Tx-Ty\|,\quad \forall x,y \in  \mathcal{X}.
\end{align*} 
Then we prove that $T$ is an invertible Lipschitz map. This improves 25 years  old Casazza-Kalton-Christensen-van Eijndhoven  perturbation. It also improves  28 years old Soderlind-Campanato perturbation and 2 years old Barbagallo-Ernst-Thera perturbation. We  give applications to the theory of    metric frames. The notion of Lipschitz atomic decomposition for Banach spaces is also introduced.




\section{Introduction}
Let $ \mathcal{X}$ be a Banach space and $I_\mathcal{X}$ be the identity operator on  $\mathcal{X}$. Carl Neumann's classical result says that if $ T : \mathcal{X}\rightarrow \mathcal{X}$ is a bounded linear operator such that $\|T-I_\mathcal{X}\|<1$, then $T$ is invertible \cite{CARLNEUMANN}. Following two results are consequences of this result. They are known as \textbf{Paley-Wiener theorems}.
\begin{enumerate}
	\item Sequences  close to orthonormal bases in Hilbert spaces are Riesz bases \cite{YOUNG, PALEYWIENER}.
	\item Sequences close to Schauder bases in Banach spaces are Schauder bases \cite{BOAS, SHAFKE}.
\end{enumerate}
History of Paley-Wiener theorems are nicely presented in \cite{ARSOVE, RETHERFORD}. It was in the setting of Hilbert spaces, Paley-Wiener theorem was first generalized by Pollard \cite{POLLARD}, second generalized  by Sz. Nagy \cite{NAGY} and third generalized by Hilding \cite{HILDING}. Hilding proved the following theorem.
\begin{theorem}\cite{HILDING}\label{HILDINGTHEOREM}
(\textbf{Hilding perturbation}) Let $\mathcal{H}$ be a Hilbert space. If  a  linear  operator $ T : \mathcal{H}\rightarrow \mathcal{H}$ is  such that there exists  $  \lambda \in \left [0, 1  \right )$ with 
\begin{align*} 
\|Th-h\|\leq\lambda\|Th\|+\lambda\|h\|,\quad \forall h \in  \mathcal{H},
\end{align*} 
then $ T $ is bounded,   invertible and 

	\begin{align*} &\frac{1-\lambda}{1+\lambda}\|h\|\leq\|Th\|\leq\frac{1+\lambda}{1-\lambda} \|h\|, \quad\forall h \in  \mathcal{H};\\
& \frac{1-\lambda}{1+\lambda}\|h\|\leq\|T^{-1}h\|\leq\frac{1+\lambda}{1-\lambda} \|h\|, \quad\forall h \in  \mathcal{H}.
\end{align*} 	
\end{theorem}
It took around 50 years to improve Theorem \ref{HILDINGTHEOREM} to the most generality for Banach spaces. 
 \begin{theorem}\cite{CASAZZAKALTON, VANEIJNDHOVEN, CASAZZACHRISTENSEN}\label{CASAZZAKALTONTHEOREM}
 (\textbf{Casazza-Kalton-Christensen-van Eijndhoven perturbation})	Let $ \mathcal{X}, \mathcal{Y}$ be Banach spaces and $ S : \mathcal{X}\rightarrow \mathcal{Y}$ be a bounded invertible operator. If  a  linear  operator $ T : \mathcal{X}\rightarrow \mathcal{Y}$ is  such that there exist  $  \lambda_1,\lambda_2 \in \left [0, 1  \right )$ with 
 \begin{align*} 
 	 \|Tx-Sx\|\leq\lambda_1\|Sx\|+\lambda_2\|Tx\|,\quad \forall x \in  \mathcal{X},
 \end{align*} 
 	then $ T $ is bounded,   invertible and 
 	\begin{align*} &\frac{1-\lambda_1}{1+\lambda_2}\|Sx\|\leq\|Tx\|\leq\frac{1+\lambda_1}{1-\lambda_2} \|Sx\|, \quad\forall x \in  \mathcal{X};\\
 	& \frac{1-\lambda_2}{1+\lambda_1}\frac{1}{\|S\|}\|y\|\leq\|T^{-1}y\|\leq\frac{1+\lambda_2}{1-\lambda_1} \|S^{-1}\|\|y\|, \quad\forall y \in  \mathcal{Y}.
 	\end{align*} 
 \end{theorem} 
  There is an improvement of Theorem \ref{CASAZZAKALTONTHEOREM} which is due to Guo with an extra assumption that $T$ is bounded.
  \begin{theorem}\cite{GUO}\label{GUOTHEOREM}
 	Let $ \mathcal{X}, \mathcal{Y}$ be Banach spaces and $ S : \mathcal{X}\rightarrow \mathcal{Y}$ be a bounded invertible operator. If  a bounded  linear  operator $ T : \mathcal{X}\rightarrow \mathcal{Y}$ is  such that there exist  $  \lambda_1 \in \left [0, 1  \right )$ and $  \lambda_2 \in \left [0, 1  \right ]$ with 
 \begin{align*} 
 \|Tx-Sx\|\leq\lambda_1\|Sx\|+\lambda_2\|Tx\|,\quad \forall x \in  \mathcal{X},
 \end{align*} 
 then $ T $ is   invertible. Further,   for every $\varepsilon>0$ satisfying $1>\lambda_2-\varepsilon>0$ and $\lambda_1+\varepsilon \|TS^{-1}\|<1$, we have 
 \begin{align*} &\frac{1-\lambda_1-\varepsilon \|TS^{-1}\|}{1+\lambda_2-\varepsilon}\|Sx\|\leq\|Tx\|\leq\frac{1+\lambda_1+\varepsilon \|TS^{-1}\|}{1-\lambda_2+\varepsilon} \|Sx\|, \quad\forall x \in  \mathcal{X};\\
 & \frac{1-\lambda_2+\varepsilon}{1+\lambda_1+\varepsilon \|TS^{-1}\|}\frac{1}{\|S\|}\|y\|\leq\|T^{-1}y\|\leq\frac{1+\lambda_2-\varepsilon}{1-\lambda_1-\varepsilon \|TS^{-1}\|} \|S^{-1}\|\|y\|, \quad\forall y \in  \mathcal{Y}.
 \end{align*}  	
  \end{theorem}

  Theorem \ref{CASAZZAKALTONTHEOREM} and its variants are useful in various studies such as stability of frames for Hilbert spaces \cite{CASAZZACHRISTENSEN}, stability of frames and atomic decompositions for Banach spaces \cite{STOEVA}, stability of frames  for Hilbert C*-modules \cite{HANJING},   stability of G-frames \cite{SUNSTABILITY}, multipliers for Hilbert  spaces \cite{STOEVABALAZS}, quantum detection problem \cite{BOTELHOANDRADE}, continuous frames \cite{GABARDOHAN}, fusion frames \cite{CASAZZAKUTYNIOK}, operator representations of frames (dynamics of frames) \cite{CHRISTENSENHASANNASAB},  pseudo-inverses of operators \cite{DING}, outer inverses of operators \cite{YANGWANG}, shift-invariant spaces \cite{KOOLIM},  frame sequences \cite{CHRISTENSENLENNARD}, sampling \cite{ZHAOCASAZZA} etc.\\
  The main objective of this paper is to generalize Theorem \ref{CASAZZAKALTONTHEOREM} for Lipschitz functions between Banach spaces. We do this in Theorem \ref{IMPORTANTTHEOREM}. We show that our result generalizes  Soderlind-Campanato Perturbation (Theorem \ref{SODERLINDCAMPANATOAPP}) and Barbagallo-Ernst-Thera perturbation (Theorem \ref{BARBAGALLOAPP}). We then give an application to the theory of frames for metric spaces. Further, the notion of Lipschitz atomic decomposition for Banach spaces is introduced and a perturbation result is derived using Theorem \ref{IMPORTANTTHEOREM}.

  \section{Improving Casazza-Kalton-Christensen-van Eijndhoven perturbation}
   Let $\mathcal{M}$ 
  be  a metric space and $\mathcal{X}$ be a Banach space. Recall that a   function $f:\mathcal{M}  \rightarrow
  \mathcal{X}$ is said to be Lipschitz if there exists $b> 0$ such that 
  \begin{align*}
  \|f(x)- f(y)\| \leq b\, d(x,y), \quad \forall x, y \in \mathcal{M}.
  \end{align*}
  A Lipschitz function $f:\mathcal{M}  \rightarrow
  \mathcal{X}$ is said to be bi-Lipschitz if there exists $a> 0$ such that 
  \begin{align*}
  a\, d(x,y) \leq \|f(x)- f(y)\| , \quad \forall x, y \in \mathcal{M}.
  \end{align*}
  \begin{definition}\cite{WEAVER}
  	Let	$\mathcal{X}$ be a Banach space.
  	\begin{enumerate}[\upshape(i)]
  		\item Let $\mathcal{M}$ be a  metric space. The collection 	$\operatorname{Lip}(\mathcal{M}, \mathcal{X})$
  		is defined as $\operatorname{Lip}(\mathcal{M}, \mathcal{X})\coloneqq \{f:\mathcal{M} 
  		\rightarrow \mathcal{X}  \operatorname{ is ~ Lipschitz} \}.$ For $f \in \operatorname{Lip}(\mathcal{M}, \mathcal{X})$, the Lipschitz number 
  		is defined as 
  		\begin{align*}
  		\operatorname{Lip}(f)\coloneqq \sup_{x, y \in \mathcal{M}, x\neq
  			y} \frac{\|f(x)-f(y)\|}{d(x,y)}.
  		\end{align*}
  		\item Let $(\mathcal{M}, 0)$ be a pointed metric space. The collection 	$\operatorname{Lip}_0(\mathcal{M}, \mathcal{X})$
  		is defined as $\operatorname{Lip}_0(\mathcal{M}, \mathcal{X})\coloneqq \{f:\mathcal{M} 
  		\rightarrow \mathcal{X}  \operatorname{ is ~ Lipschitz ~ and } f(0)=0\}.$
  		For $f \in \operatorname{Lip}_0(\mathcal{M}, \mathcal{X})$, the Lipschitz norm
  		is defined as 
  		\begin{align*}
  		\|f\|_{\operatorname{Lip}_0}\coloneqq \sup_{x, y \in \mathcal{M}, x\neq
  			y} \frac{\|f(x)-f(y)\|}{d(x,y)}.
  		\end{align*}
  	\end{enumerate}
  	
  \end{definition}
  \begin{theorem}\cite{WEAVER}\label{BANACHALGEBRA}
  	Let	$\mathcal{X}$ be a Banach space.
  	\begin{enumerate}[\upshape(i)]
  		\item If $\mathcal{M}$ is a  metric space, then 	$\operatorname{Lip}(\mathcal{M},
  		\mathcal{X})$ is a semi-normed vector  space w.r.t. the semi-norm  $\operatorname{Lip}(\cdot)$.
  		\item If $(\mathcal{M}, 0)$ is a pointed metric space, then 	$\operatorname{Lip}_0(\mathcal{M},
  		\mathcal{X})$ is a Banach space w.r.t. the  norm
  		$\|\cdot\|_{\operatorname{Lip}_0}$. Further, $\operatorname{Lip}_0(\mathcal{X})\coloneqq\operatorname{Lip}_0(\mathcal{X},
  		\mathcal{X})$ is a unital Banach algebra. In particular, if $T \in \operatorname{Lip}_0(\mathcal{X})$ satisfies $
  		\|T-I_\mathcal{X}\|_{\operatorname{Lip}_0}<1,$ 
  		then $T $ is invertible and $T^{-1} \in \operatorname{Lip}_0(\mathcal{X})$.
  	\end{enumerate}
  \end{theorem}	
  We now develop perturbation result for Lipschitz functions using various results. Our developments are motivated from the linear version of improvement of Paley-Wiener theorem by van Eijndhoven \cite{VANEIJNDHOVEN}. Since $ \text{Lip}_0(\mathcal{X})$ is a unital Banach algebra, we can talk about the notion of spectrum and resolvents. In the remaining part of the paper, the spectrum of an element $ T\in \text{Lip}_0(\mathcal{X})$ is denoted by $\sigma(T)$ and the resolvent by $\rho(T).$
  \begin{lemma}\label{FIRSTLEMMA}
  	Let  $ \mathcal{X}$ be a Banach space and $ T\in \text{Lip}_0(\mathcal{X})$. Suppose there are $\alpha_0\in \mathbb{R}$ and $ \beta_0>0$  such that 
  	\begin{align}\label{FIRSTLEMMAINEQUALITY}
  	\|Tx-Ty-\alpha(x-y)\|\geq \beta_0\|x-y\|, \quad \forall x,y \in \mathcal{X}, \forall \alpha \leq \alpha_0.
  	\end{align}
  	Then $(-\infty, \alpha_0]\subseteq \rho(T).$
  \end{lemma}
  \begin{proof}
If $\sigma(T)\cap\mathbb{R}=\emptyset$, then there is nothing to argue. So let $\sigma(T)\cap\mathbb{R}\neq \emptyset$. Since $\sigma(T)$ is non empty, define
\begin{align*}
\lambda_0\coloneqq \min\{\lambda\in \mathbb{R}: \lambda \in \sigma(T)\cap \mathbb{R}\}.
\end{align*} 
Claim: $\lambda_0\geq \alpha_0$. If this is false, we have $\lambda_0< \alpha_0$. From the assumption we then have 
	\begin{align*}
\|Tx-Ty-\lambda_0(x-y)\|\geq \beta_0\|x-y\|, \quad \forall x,y \in \mathcal{X} 
\end{align*}
which implies 
\begin{align*}
\|(T-\lambda_0I_\mathcal{X})x-(T-\lambda_0I_\mathcal{X})y\|\geq \beta_0\|x-y\|, \quad \forall x,y \in \mathcal{X} 
\end{align*}
which says that $T-\lambda_0I_\mathcal{X}$ is injective. We now try to show that $T-\lambda_0I_\mathcal{X}$ is surjective. Let $y \in \mathcal{X}$. Define 
\begin{align*}
\alpha_n\coloneqq \lambda_0-\frac{1}{n}, \quad \forall n \in \mathbb{N}.
\end{align*}
From the definition of $\lambda_0$ we then have $T-\alpha_nI_\mathcal{X}$ is invertible for all  $n \in \mathbb{N}$. Using Inequality (\ref{FIRSTLEMMAINEQUALITY}) we  have 

\begin{align*}
\|(T-\alpha_nI_\mathcal{X})x-(T-\alpha_nI_\mathcal{X})y\|\geq \beta_0\|x-y\|, \quad \forall x,y \in \mathcal{X}, \forall n \in \mathbb{N} 
\end{align*}
which gives 
\begin{align*}
\|u-v\|\geq \beta_0\|(T-\alpha_nI_\mathcal{X})^{-1}u-(T-\alpha_nI_\mathcal{X})^{-1}v\|, \quad \forall u,v \in \mathcal{X}, \forall n \in \mathbb{N} 
\end{align*}
which implies 
\begin{align*}
\|(T-\alpha_nI_\mathcal{X})^{-1}\|_{\text{Lip}_0}\leq \frac{1}{\beta_0}, \quad  \forall n \in \mathbb{N}.
\end{align*}
Now define 
\begin{align*}
x_n\coloneqq (T-\alpha_nI_\mathcal{X})^{-1}y, \quad \forall n \in \mathbb{N}.
\end{align*}
Now using resolvent identity in Banach algebra  \cite{DALES} we get 
\begin{align*}
\|x_n-x_m\|&=\left|\frac{1}{n}-\frac{1}{m}\right|\|(T-\alpha_nI_\mathcal{X})^{-1}(T-\alpha_mI_\mathcal{X})^{-1}y\|\\
&\leq \left|\frac{1}{n}-\frac{1}{m}\right|\|(T-\alpha_nI_\mathcal{X})^{-1}\|_{\text{Lip}_0}\|(T-\alpha_mI_\mathcal{X})^{-1}\|_{\text{Lip}_0}\|y\|\\
&\leq \left|\frac{1}{n}-\frac{1}{m}\right|\frac{1}{\beta_0^2}\|y\|, \quad \forall n,m \in \mathbb{N}.
\end{align*}
Thus $\{x_n\}_n$ is Cauchy and it converges to an element, say $x \in \mathcal{X}$. Finally 
\begin{align*}
\|(T-\lambda_0I_\mathcal{X})x-y\|&=\|(T-\lambda_0I_\mathcal{X})x-(T-\alpha_nI_\mathcal{X})x_n\|\\
&\leq \|(T-\lambda_0I_\mathcal{X})x-(T-\lambda_0I_\mathcal{X})x_n\|+\|(T-\lambda_0I_\mathcal{X})x_n-(T-\alpha_nI_\mathcal{X})x_n\|\\
&\leq \|(T-\lambda_0I_\mathcal{X})^{-1}\|_{\text{Lip}_0}\|x_n-x\|+|\alpha_n-\lambda_0|\|x_n\|\\
&\leq \|(T-\lambda_0I_\mathcal{X})^{-1}\|_{\text{Lip}_0}\|x_n-x\|+\frac{1}{n}\|x_n\|, \quad \forall n \in \mathbb{N}.
\end{align*}
Letting $n\to \infty$, we get $(T-\lambda_0I_\mathcal{X})x=y$. Thus $T-\lambda_0I_\mathcal{X} $ is invertible. Hence $\lambda_0\in \sigma(T)\cap\rho(T)=\emptyset$ which is a contradiction which proves the claim as well as the lemma.
\end{proof}
\begin{remark}
Note that Lemma \ref{FIRSTLEMMA} also holds for real Banach spaces.  In fact, if spectrum   $\sigma(T)$ is empty then the lemma holds trivially and if it is non empty, then it becomes a compact set  and proof of Lemma \ref{FIRSTLEMMA} can be carried over. Due to this, all coming results in this paper are valid for real Banach spaces also.
\end{remark}
  \begin{theorem}\label{UNIONTHEOREM}
  		Let  $ \mathcal{X}$ be a Banach space and $ T\in \text{Lip}_0(\mathcal{X})$. Suppose there are $\alpha_1\in \mathbb{R}$ and $ \beta_1>0$  such that 
  \begin{align}\label{FIRSTTHMINEQUALITY}
  \|Tx-Ty-\alpha(x-y)\|\geq \beta_1\|x-y\|, \quad \forall x,y \in \mathcal{X}, \forall \alpha \leq \alpha_1.
  \end{align}
  Then $(-\infty, \alpha_1+\beta_1)\subseteq \rho(T).$	
  \end{theorem}
   \begin{proof}
  	Let $0<\eta<\beta_1$. Define $\alpha_0\coloneqq \alpha_1+\eta$ and $\beta_0\coloneqq \beta_1-\eta$. Let $\alpha\leq \alpha_0$. 
  	
  	Case (i): $\alpha\leq  \alpha_1$. Then from Inequality (\ref{FIRSTTHMINEQUALITY}), $\|Tx-Ty-\alpha(x-y)\|\geq \beta_1\|x-y\|,  \forall x,y \in \mathcal{X}$.
  	
  	 Case (ii):
  	Let $\alpha_1< \alpha\leq \alpha_0$. Then 
  	\begin{align*}
  	\|Tx-Ty-\alpha(x-y)\|&\geq 	\|Tx-Ty-\alpha_1(x-y)\|-(\alpha-\alpha_1)\|x-y\|\\
  	&\geq(\beta_1-\eta)\|x-y\|=\beta_0\|x-y\|, \quad \forall x,y \in \mathcal{X}.
  	\end{align*}
  	Lemma \ref{FIRSTLEMMA} now gives $(-\infty, \alpha_0]\subseteq \rho(T).$ Result follows by observing that 
  	\begin{align*}
  	(-\infty, \alpha_1+\beta_1)=\bigcup\limits_{0<\eta<\beta_1}(-\infty, \alpha_1+\eta].
  	\end{align*}
  \end{proof}
  \begin{theorem}\label{OIDENTITYTHEOREM}
  		Let $ \mathcal{X}$ be a Banach space,   $ T : \mathcal{X}\rightarrow \mathcal{X}$ be a map,   $T0=0$ and there exist  $ \lambda_1,\lambda_2 \in \left [0, 1  \right )$ such that 
  	\begin{align} \label{PER}
  		 \|Tx-Ty-(x-y)\|\leq\lambda_1\|x-y\|+\lambda_2\|Tx-Ty\|,\quad \forall x,y \in  \mathcal{X}.
  		\end{align} 
  		Then 
  		\begin{enumerate}[\upshape(i)]
  			\item $T$ is Lipschitz and 
  			\begin{align}\label{PERE}
  			\frac{1-\lambda_1}{1+\lambda_2}\|x-y\|\leq\|Tx-Ty\|\leq\frac{1+\lambda_1}{1-\lambda_2} \|x-y\|, \quad \forall x,y \in \mathcal{X}.
  			\end{align}
  			\item We have 
  			\begin{align*}
  			\left(-\infty, \frac{1-\lambda_1}{1+\lambda_2}\right)\subseteq \rho(T).
  			\end{align*}
  			\item $T$ is invertible and 
  			\begin{align*}
  			\frac{1-\lambda_2}{1+\lambda_1}\|x-y\|\leq\|T^{-1}x-T^{-1}y\|\leq\frac{1+\lambda_2}{1-\lambda_1} \|x-y\|, \quad\forall x,y \in  \mathcal{X}.
  			\end{align*}
  			\item We have 
  			\begin{align*}
  			\frac{1-\lambda_1}{1+\lambda_2}\leq\|T\|_{\text{Lip}_0}\leq\frac{1+\lambda_1}{1-\lambda_2}  \quad \text{ and } \quad \frac{1-\lambda_2}{1+\lambda_1}\leq\|T^{-1}\|_{\text{Lip}_0}\leq\frac{1+\lambda_2}{1-\lambda_1}. 
  			\end{align*}
  		\end{enumerate}
  \end{theorem}
  \begin{proof}
 Let  $x,y \in \mathcal{X}.$	Then using Inequality (\ref{PER}) 
 \begin{align*}
 \|Tx-Ty\|&\leq  \|Tx-Ty-(x-y)\|+\|x-y\|\leq \lambda_1\|x-y\|+\lambda_2\|Tx-Ty\|+\|x-y\|\\
 &=(1+\lambda_1)\|x-y\|+\lambda_2\|Tx-Ty\|\\
 &\implies \|Tx-Ty\|\leq\frac{1+\lambda_1}{1-\lambda_2} \|x-y\|
 \end{align*}
 and 
 \begin{align*}
 \|x-y\|&\leq \|Tx-Ty-(x-y)\|+\|Tx-Ty\|\leq \lambda_1\|x-y\|+\lambda_2\|Tx-Ty\|+\|Tx-Ty\|\\
 &=\lambda_1\|x-y\|+(1+\lambda_2)\|Tx-Ty\|\\
 &\implies \frac{1-\lambda_1}{1+\lambda_2}\|x-y\|\leq\|Tx-Ty\|.
 \end{align*}
 Hence $T$ is Lipschitz and (i) holds. Let $\alpha\leq0$.  Then 
 \begin{align*}
 \|Tx-Ty-\alpha(x-y)\|&=\|(1-\alpha)(x-y)-(x-y-(Tx-Ty))\|\\
 &\geq (1-\alpha)\|x-y\|-\|Tx-Ty-(x-y)\|\\
 &\geq (1-\alpha)\|x-y\|-\lambda_1\|x-y\|-\lambda_2\|Tx-Ty\|\\
 &= (1-\alpha-\lambda_1)\|x-y\|-\lambda_2\|Tx-Ty\|\\
 &\geq (1-\alpha-\lambda_1)\|x-y\|-\lambda_2\|Tx-Ty-\alpha(x-y)\|+\lambda_2\alpha \|x-y\|\\
 &=(1-\alpha-\lambda_1+\lambda_2\alpha)\|x-y\|-\lambda_2\|Tx-Ty-\alpha(x-y)\|
 \end{align*}
 which implies 
 
 \begin{align*}
 \|Tx-Ty-\alpha(x-y)\|&\geq \frac{1-\alpha-\lambda_1+\lambda_2\alpha}{1+\lambda_2}\|x-y\|\\
 &=\frac{1-\lambda_1-(1-\lambda_2)\alpha}{1+\lambda_2}\|x-y\|\geq \frac{1-\lambda_1}{1+\lambda_2}\|x-y\|.
 \end{align*}
 By applying Theorem \ref{UNIONTHEOREM} for $\alpha_1=0$ and $\beta_1=\frac{1-\lambda_1}{1+\lambda_2}$ we get $	
(-\infty, \frac{1-\lambda_1}{1+\lambda_2})\subseteq \rho(T).$ Since $0\in (-\infty, \frac{1-\lambda_1}{1+\lambda_2})$, $T $ is invertible. Using Inequality (\ref{PERE})  we then get 
 \begin{align*}
 	\frac{1-\lambda_1}{1+\lambda_2}\|T^{-1}u-T^{-1}v\|\leq\|u-v\|\leq\frac{1+\lambda_1}{1-\lambda_2} \|T^{-1}u-T^{-1}v\|, \quad \forall u,v \in \mathcal{X}
 \end{align*}
 which gives (iii). Finally (iv) follows from (i) and (iii).
  \end{proof}

\begin{theorem}\label{OSECONDTHEOREM}
	Let $ \mathcal{X}$, $ \mathcal{Y}$ be  Banach spaces and $S\in \text{Lip}_0(\mathcal{X}, \mathcal{Y}) $ be invertible. Let     $ T : \mathcal{X}\rightarrow \mathcal{Y}$ be a map,   $T0=0$ and there exist  $ \lambda_1,\lambda_2 \in \left [0, 1  \right )$ such that 
	\begin{align}\label{123}
	\|Tx-Ty-(Sx-Sy)\|\leq\lambda_1\|Sx-Sy\|+\lambda_2\|Tx-Ty\|,\quad \forall x,y \in  \mathcal{X}.
	\end{align} 
	Then 
	\begin{enumerate}[\upshape(i)]
		\item $T$ is Lipschitz and 
		\begin{align*}
		\frac{1-\lambda_1}{1+\lambda_2}\|Sx-Sy\|\leq\|Tx-Ty\|\leq\frac{1+\lambda_1}{1-\lambda_2} \|Sx-Sy\|, \quad \forall x,y \in \mathcal{X}.
		\end{align*}
		\item $\alpha S-T$ is invertible for all $\alpha \in \left(-\infty, \frac{1-\lambda_1}{1+\lambda_2}\right)$.
		\item $T$ is invertible and 
		\begin{align*}
		\frac{1-\lambda_2}{1+\lambda_1}\frac{1}{\|S\|_{\text{Lip}_0}}\|u-v\|\leq\|T^{-1}u-T^{-1}v\|\leq\frac{1+\lambda_2}{1-\lambda_1} \|S^{-1}\|_{\text{Lip}_0}\|u-v\|, \quad\forall u,v \in  \mathcal{Y}.
		\end{align*}
		\item We have 
		\begin{align*}
		\frac{1-\lambda_1}{1+\lambda_2} \|S\|_{\text{Lip}_0}\leq\|T\|_{\text{Lip}_0}\leq \frac{1+\lambda_1}{1-\lambda_2}\|S\|_{\text{Lip}_0}  \quad \text{ and } \quad \frac{1-\lambda_2}{1+\lambda_1}\frac{1}{\|S\|_{\text{Lip}_0}}\leq\|T^{-1}\|_{\text{Lip}_0}\leq\frac{1+\lambda_2}{1-\lambda_1}\|S^{-1}\|_{\text{Lip}_0} . 
		\end{align*}
	\end{enumerate}
\end{theorem}
\begin{proof}
Define $R\coloneqq TS^{-1}$. Then Inequality (\ref{123}) gives 
\begin{align*}
\|TS^{-1}u-TS^{-1}v-(SS^{-1}u-SS^{-1}v)\|\leq \lambda_1\|SS^{-1}u-SS^{-1}v\|+\lambda_2\|TS^{-1}u-TS^{-1}v\|, \quad \forall u,v \in  \mathcal{Y},
\end{align*}
i.e., 
\begin{align*}
\|Ru-Rv-(u-v)\|\leq \lambda_1\|u-v\|+\lambda_2\|Ru-Rv\|, \quad \forall u,v \in  \mathcal{Y}.
\end{align*}
By applying Theorem \ref{OIDENTITYTHEOREM} to $R$ we get the following.
\begin{enumerate}[\upshape(i)]
	\item $R$ is Lipschitz hence $T$ is Lipschitz. Further, 
		\begin{align*}
	\frac{1-\lambda_1}{1+\lambda_2}\|Sx-Sy\|\leq\|R(Sx)-R(Sy)\|\leq\frac{1+\lambda_1}{1-\lambda_2} \|Sx-Sy\|, \quad \forall x,y \in \mathcal{X}.
	\end{align*}
	But  $\|R(Sx)-R(Sy)\|=\|Tx-Ty\|$, $\forall x,y \in \mathcal{X}.$
	\item $\alpha I_\mathcal{X}-R$ is invertible for all $ \alpha \in 	\left(-\infty, \frac{1-\lambda_1}{1+\lambda_2}\right)$. Since $S$ is invertible we then have $\alpha S-T$ is invertible for all $\alpha \in \left(-\infty, \frac{1-\lambda_1}{1+\lambda_2}\right)$.
	\item $R$ is invertible hence $T$ is invertible. Further, 
	
		\begin{align*}
	\frac{1-\lambda_2}{1+\lambda_1}\frac{1}{\|S\|_{\text{Lip}_0}}\|u-v\|&\leq	\frac{1}{\|S\|_{\text{Lip}_0}}\|R^{-1}u-R^{-1}v\|
	\leq \|S^{-1}(R^{-1}u)-S^{-1}(R^{-1}v)\|\\
	&=	\|T^{-1}u-T^{-1}v\|\leq \|S^{-1}\|_{\text{Lip}_0}\|R^{-1}u-R^{-1}v\|\\
	&\leq \frac{1+\lambda_2}{1-\lambda_1}\|S^{-1}\|_{\text{Lip}_0}\|u-v\|, \quad\forall u,v \in  \mathcal{Y}.
	\end{align*}
	\item This follows easily from (i) and (iii).
\end{enumerate}
\end{proof}
Our next job is to derive the results by removing  the condition $T0=0$.
\begin{theorem}\label{NOBASEPOINTFIRST}
Let $ \mathcal{X}$ be a Banach space,   $ T : \mathcal{X}\rightarrow \mathcal{X}$ be a map  and there exist  $ \lambda_1,\lambda_2 \in \left [0, 1  \right )$ such that 
\begin{align*} 
\|Tx-Ty-(x-y)\|\leq\lambda_1\|x-y\|+\lambda_2\|Tx-Ty\|,\quad \forall x,y \in  \mathcal{X}.
\end{align*} 
Then 
\begin{enumerate}[\upshape(i)]
	\item $T$ is Lipschitz and 
	\begin{align*}
	\frac{1-\lambda_1}{1+\lambda_2}\|x-y\|\leq\|Tx-Ty\|\leq\frac{1+\lambda_1}{1-\lambda_2} \|x-y\|, \quad \forall x,y \in \mathcal{X}.
	\end{align*}
	\item $\alpha I_\mathcal{X}-T$ is invertible for all $\alpha \in \left(-\infty, \frac{1-\lambda_1}{1+\lambda_2}\right)$.
\item $T$ is invertible and 
	\begin{align*}
	\frac{1-\lambda_2}{1+\lambda_1}\|x-y\|\leq\|T^{-1}x-T^{-1}y\|\leq\frac{1+\lambda_2}{1-\lambda_1} \|x-y\|, \quad\forall x,y \in  \mathcal{X}.
	\end{align*}
	\item We have 
	\begin{align*}
	\frac{1-\lambda_1}{1+\lambda_2}\leq \text{Lip}(T)\leq\frac{1+\lambda_1}{1-\lambda_2}  \quad \text{ and } \quad \frac{1-\lambda_2}{1+\lambda_1}\leq\text{Lip}(T^{-1})\leq\frac{1+\lambda_2}{1-\lambda_1}. 
	\end{align*}
\end{enumerate}
\end{theorem}
\begin{proof}
Define 
\begin{align*}
\tilde{T} x \coloneqq Tx-T0, \quad \forall x \in \mathcal{X}.
\end{align*}	
Then $\tilde{T}0=0$ and
\begin{align*}
\|\tilde{T}x-\tilde{T}y-(x-y)\|&=\|Tx-Ty-(x-y)\|\leq\lambda_1\|x-y\|+\lambda_2\|Tx-Ty\|\\
&=\lambda_1\|x-y\|+\lambda_2\|\tilde{T}x-\tilde{T}y\|,\quad \forall x,y \in  \mathcal{X}.
\end{align*}
Applying Theorem \ref{OIDENTITYTHEOREM} and using the fact that `a map is bijective if and only if its translate is bijective', proof is complete.
\end{proof}
\begin{theorem}\label{IMPORTANTTHEOREM}
	Let $ \mathcal{X}$, $ \mathcal{Y}$ be  Banach spaces and $S\in \text{Lip}(\mathcal{X}, \mathcal{Y}) $ be invertible. Let     $ T : \mathcal{X}\rightarrow \mathcal{Y}$ be a map and there exist  $ \lambda_1,\lambda_2 \in \left [0, 1  \right )$ such that 
	\begin{align*}
	\|Tx-Ty-(Sx-Sy)\|\leq\lambda_1\|Sx-Sy\|+\lambda_2\|Tx-Ty\|,\quad \forall x,y \in  \mathcal{X}.
	\end{align*} 
	Then 
	\begin{enumerate}[\upshape(i)]
		\item $T$ is Lipschitz and 
		\begin{align*}
		\frac{1-\lambda_1}{1+\lambda_2}\|Sx-Sy\|\leq\|Tx-Ty\|\leq\frac{1+\lambda_1}{1-\lambda_2} \|Sx-Sy\|, \quad \forall x,y \in \mathcal{X}.
		\end{align*}
		\item $\alpha S-T$ is invertible for all $\alpha \in \left(-\infty, \frac{1-\lambda_1}{1+\lambda_2}\right)$.
		\item $T$ is invertible and 
		\begin{align*}
		\frac{1-\lambda_2}{1+\lambda_1}\frac{1}{\text{Lip}(S)}\|u-v\|\leq\|T^{-1}u-T^{-1}v\|\leq\frac{1+\lambda_2}{1-\lambda_1} \text{Lip}(S^{-1})\|u-v\|, \quad\forall u,v \in  \mathcal{Y}.
		\end{align*}
		\item We have 
		\begin{align*}
		\frac{1-\lambda_1}{1+\lambda_2} \text{Lip}(S)\leq\text{Lip}(T)\leq \frac{1+\lambda_1}{1-\lambda_2}\text{Lip}(S)  \quad \text{ and } \quad \frac{1-\lambda_2}{1+\lambda_1}\frac{1}{\text{Lip}(S)}\leq\text{Lip}(T^{-1})\leq\frac{1+\lambda_2}{1-\lambda_1}\text{Lip}(S^{-1}) . 
		\end{align*}
	\end{enumerate}
\end{theorem}
\begin{proof}
Define	$R\coloneqq TS^{-1}$ and the proof is similar to proof of Theorem \ref{OSECONDTHEOREM}.
\end{proof}
Following two corollaries are motivated from \cite{HILDING}.
\begin{corollary}
Let $p\geq1$. 	Let $ \mathcal{X}$, $ \mathcal{Y}$ be  Banach spaces and $S\in \text{Lip}(\mathcal{X}, \mathcal{Y}) $ be invertible. Let     $ T : \mathcal{X}\rightarrow \mathcal{Y}$ be a map and there exist  $ \lambda_1,\lambda_2 \in \left [0, 1  \right )$ such that 
\begin{align}\label{PGREATER1INEQUALITY}
\|Tx-Ty-(Sx-Sy)\|\leq((\lambda_1\|Sx-Sy\|)^p+(\lambda_2\|Tx-Ty\|)^p)^\frac{1}{p},\quad \forall x,y \in  \mathcal{X}.
\end{align} 
Then 
\begin{enumerate}[\upshape(i)]
	\item $T$ is Lipschitz and 
	\begin{align*}
	\frac{1-\lambda_1}{1+\lambda_2}\|Sx-Sy\|\leq\|Tx-Ty\|\leq\frac{1+\lambda_1}{1-\lambda_2} \|Sx-Sy\|, \quad \forall x,y \in \mathcal{X}.
	\end{align*}
	\item $\alpha S-T$ is invertible for all $\alpha \in \left(-\infty, \frac{1-\lambda_1}{1+\lambda_2}\right)$.
	\item $T$ is invertible and 
	\begin{align*}
	\frac{1-\lambda_2}{1+\lambda_1}\frac{1}{\text{Lip}(S)}\|u-v\|\leq\|T^{-1}u-T^{-1}v\|\leq\frac{1+\lambda_2}{1-\lambda_1} \text{Lip}(S^{-1})\|u-v\|, \quad\forall u,v \in  \mathcal{Y}.
	\end{align*}
	\item We have 
	\begin{align*}
	\frac{1-\lambda_1}{1+\lambda_2} \text{Lip}(S)\leq\text{Lip}(T)\leq \frac{1+\lambda_1}{1-\lambda_2}\text{Lip}(S)  \quad \text{ and } \quad \frac{1-\lambda_2}{1+\lambda_1}\frac{1}{\text{Lip}(S)}\leq\text{Lip}(T^{-1})\leq\frac{1+\lambda_2}{1-\lambda_1}\text{Lip}(S^{-1}) . 
	\end{align*}
\end{enumerate}	
\end{corollary}
\begin{proof}
Note that if $r,s\geq0 $, then 
\begin{align*}
(r^p+s^p)^\frac{1}{p}\leq r+s \quad \text{if} \quad p\geq 1.
\end{align*}
Hence Inequality (\ref{PGREATER1INEQUALITY}) gives 
\begin{align*}
\|Tx-Ty-(Sx-Sy)\|&\leq((\lambda_1\|Sx-Sy\|)^p+(\lambda_2\|Tx-Ty\|)^p)^\frac{1}{p}\\
&\leq \lambda_1\|Sx-Sy\|+\lambda_2\|Tx-Ty\|,\quad \forall x,y \in  \mathcal{X}.
\end{align*}
Result follows by applying  Theorem \ref{IMPORTANTTHEOREM}.
\end{proof}
\begin{corollary}
	Let $0<p<1$. 	Let $ \mathcal{X}$, $ \mathcal{Y}$ be  Banach spaces and $S\in \text{Lip}(\mathcal{X}, \mathcal{Y}) $ be invertible. Let     $ T : \mathcal{X}\rightarrow \mathcal{Y}$ be a map and there exist  $ \lambda_1,\lambda_2 \in \left [0 ,2^{1-\frac{1}{p}}  \right )$ such that  
	\begin{align*}
	\|Tx-Ty-(Sx-Sy)\|\leq((\lambda_1\|Sx-Sy\|)^p+(\lambda_2\|Tx-Ty\|)^p)^\frac{1}{p},\quad \forall x,y \in  \mathcal{X}.
	\end{align*} 
	Then 
	\begin{enumerate}[\upshape(i)]
		\item $T$ is Lipschitz and 
		\begin{align*}
		\frac{1-2^{\frac{1}{p}-1}\lambda_1}{1+2^{\frac{1}{p}-1}\lambda_2}\|Sx-Sy\|\leq\|Tx-Ty\|\leq\frac{1+2^{\frac{1}{p}-1}\lambda_1}{1-2^{\frac{1}{p}-1}\lambda_2} \|Sx-Sy\|, \quad \forall x,y \in \mathcal{X}.
		\end{align*}
		\item $\alpha S-T$ is invertible for all $\alpha \in \left(-\infty, \frac{1-2^{\frac{1}{p}-1}\lambda_1}{1+2^{\frac{1}{p}-1}\lambda_2}\right)$.
		\item $T$ is invertible and 
		\begin{align*}
		\frac{1-2^{\frac{1}{p}-1}\lambda_2}{1+2^{\frac{1}{p}-1}\lambda_1}\frac{1}{\text{Lip}(S)}\|u-v\|\leq\|T^{-1}u-T^{-1}v\|\leq\frac{1+2^{\frac{1}{p}-1}\lambda_2}{1-2^{\frac{1}{p}-1}\lambda_1} \text{Lip}(S^{-1})\|u-v\|, \quad\forall u,v \in  \mathcal{Y}.
		\end{align*}
		\item We have 
		\begin{align*}
		\frac{1-2^{\frac{1}{p}-1}\lambda_1}{1+2^{\frac{1}{p}-1}\lambda_2} \text{Lip}(S)\leq\text{Lip}(T)\leq \frac{1+2^{\frac{1}{p}-1}\lambda_1}{1-2^{\frac{1}{p}-1}\lambda_2}\text{Lip}(S)  \quad \text{ and } \\
		 \frac{1-2^{\frac{1}{p}-1}\lambda_2}{1+2^{\frac{1}{p}-1}\lambda_1}\frac{1}{\text{Lip}(S)}\leq\text{Lip}(T^{-1})\leq\frac{1+2^{\frac{1}{p}-1}\lambda_2}{1-2^{\frac{1}{p}-1}\lambda_1}\text{Lip}(S^{-1}) . 
		\end{align*}
	\end{enumerate}	
\end{corollary}
\begin{proof}
Note that if $r,s\geq0 $, then 
	\begin{align*}
	 (r^p+s^p)^\frac{1}{p}\leq 2^{\frac{1}{p}-1}(r+s) \quad \text{if} \quad p< 1.
	\end{align*}
Hence Inequality (\ref{PGREATER1INEQUALITY}) gives 
\begin{align*}
\|Tx-Ty-(Sx-Sy)\|&\leq((\lambda_1\|Sx-Sy\|)^p+(\lambda_2\|Tx-Ty\|)^p)^\frac{1}{p}\\
&\leq  2^{\frac{1}{p}-1}\lambda_1\|Sx-Sy\|+ 2^{\frac{1}{p}-1}\lambda_2\|Tx-Ty\|,\quad \forall x,y \in  \mathcal{X}.
\end{align*}
Result follows by applying  Theorem \ref{IMPORTANTTHEOREM}.	
\end{proof}
Next we generalize Corollary 1 in \cite{CASAZZACHRISTENSEN}.
\begin{corollary}
 	Let $ \mathcal{X}$, $ \mathcal{Y}$ be  Banach spaces and $S\in \text{Lip}(\mathcal{X}, \mathcal{Y}) $ be invertible. Let     $ T : \mathcal{X}\rightarrow \mathcal{Y}$ be a  map and there exists  $ \lambda \in \left [0, 1  \right )$ such that 
\begin{align*}
\|Tx-Ty-(Sx-Sy)\|\leq\lambda\|Sx-Sy\|+\|Tx-Ty\|,\quad \forall x,y \in  \mathcal{X}.
\end{align*} 
Then $T$ is Lipschitz, invertible and 
\begin{align*}
 \text{Lip}(T^{-1})\leq\frac{2}{1-\lambda}\text{Lip}(S^{-1}) . 
	\end{align*}
\end{corollary}
\begin{proof}
Define $R\coloneqq TS^{-1}$. Then 
\begin{align*}
\|Ru-Rv-(u-v)\|\leq \lambda \|u-v\|+\|Ru-Rv\|, \quad \forall u,v \in  \mathcal{Y}.
\end{align*}
Note that   $\text{Lip}(R)\neq 0$. Let 
\begin{align*}
0<\varepsilon < \min \left\{1, \frac{1-\lambda}{\text{Lip}(R)}\right\}.
\end{align*}
Define 
$\lambda_1 \coloneqq \lambda -\varepsilon \text{Lip}(R)$ and $\lambda_2 \coloneqq 1-\varepsilon$. Then  $ \lambda_1,\lambda_2 \in \left [0, 1  \right )$ and 
\begin{align*}
\|Ru-Rv-(u-v)\|&\leq \lambda \|u-v\|+\|Ru-Rv\|\\
&\leq \lambda \|u-v\|+\|Ru-Rv\|+\varepsilon (\text{Lip}(R)\|u-v\|- \|Ru-Rv\|)\\
&=(\lambda+\varepsilon \text{Lip}(R))\|u-v\|+(1-\varepsilon)\|Ru-Rv\| , \quad \forall u,v \in  \mathcal{Y}.
\end{align*}
By applying Theorem \ref{NOBASEPOINTFIRST} we get that $R$ is Lipschitz, invertible  and 
\begin{align*}
\text{Lip}(R^{-1})\leq\frac{2-\varepsilon}{1-(\lambda+\varepsilon \text{Lip}(R))}.
\end{align*}
Since $\varepsilon$ can be made arbitrarily small, we must have 
\begin{align*}
\text{Lip}(R^{-1})\leq\frac{2}{1-\lambda}.
\end{align*}
Substituting the expression of $R$ gives 
\begin{align*}
\frac{1}{\text{Lip}(S^{-1})}\text{Lip}(T^{-1})\leq \text{Lip}(ST^{-1})=\text{Lip}(R^{-1})\leq\frac{2}{1-\lambda}.
\end{align*}
\end{proof}
We finally derive the following  non-linear version of Theorem \ref{GUOTHEOREM}. 
\begin{theorem}\label{LASTTHEOREM}
Let $ \mathcal{X}$, $ \mathcal{Y}$ be  Banach spaces and $S\in \text{Lip}(\mathcal{X}, \mathcal{Y}) $ be invertible. Let     $ T : \mathcal{X}\rightarrow \mathcal{Y}$ be a Lipschitz map and there exist  $  \lambda_1 \in \left [0, 1  \right )$ and $  \lambda_2 \in \left [0, 1  \right ]$ such that 
	\begin{align*}
	\|Tx-Ty-(Sx-Sy)\|\leq\lambda_1\|Sx-Sy\|+\lambda_2\|Tx-Ty\|,\quad \forall x,y \in  \mathcal{X}.
	\end{align*} 
	Then $T$ is Lipschitz invertible.  Further,   for every $\varepsilon>0$ satisfying $1>\lambda_2-\varepsilon>0$ and $\lambda_1+\varepsilon \text{Lip}(TS^{-1})<1$, we have
	\begin{enumerate}[\upshape(i)]
		\item 
		\begin{align*}
		\frac{1-\lambda_1-\varepsilon \text{Lip}(TS^{-1})}{1+\lambda_2-\varepsilon}\|Sx-Sy\|\leq\|Tx-Ty\|\leq\frac{1+\lambda_1+\varepsilon \text{Lip}(TS^{-1})}{1-\lambda_2+\varepsilon} \|Sx-Sy\|, \quad \forall x,y \in \mathcal{X}.
		\end{align*}
		\item $\alpha S-T$ is invertible for all $\alpha \in \left(-\infty, \frac{1-\lambda_1-\varepsilon \text{Lip}(TS^{-1})}{1+\lambda_2-\varepsilon}\right)$.
		\item $T$ is invertible and 
		\begin{align*}
		\frac{1-\lambda_2+\varepsilon}{1+\lambda_1+\varepsilon \text{Lip}(TS^{-1})}\frac{1}{\text{Lip}(S)}\|u-v\|\leq\|T^{-1}u-T^{-1}v\|\leq\frac{1+\lambda_2-\varepsilon}{1-\lambda_1-\varepsilon \text{Lip}(TS^{-1})} \text{Lip}(S^{-1})\|u-v\|, \quad\forall u,v \in  \mathcal{Y}.
		\end{align*}
		\item We have 
		\begin{align*}
		&\frac{1-\lambda_1-\varepsilon \text{Lip}(TS^{-1})}{1+\lambda_2-\varepsilon} \text{Lip}(S)\leq\text{Lip}(T)\leq \frac{1+\lambda_1+\varepsilon \text{Lip}(TS^{-1})}{1-\lambda_2+\varepsilon}\text{Lip}(S)  \quad  \text{ and } \\
		 & \frac{1-\lambda_2+\varepsilon}{1+\lambda_1+\varepsilon \text{Lip}(TS^{-1})}\frac{1}{\text{Lip}(S)}\leq\text{Lip}(T^{-1})\leq\frac{1+\lambda_2-\varepsilon}{1-\lambda_1-\varepsilon \text{Lip}(TS^{-1})}\text{Lip}(S^{-1}) . 
		\end{align*}
	\end{enumerate}
\end{theorem}
\begin{proof}
Define	$R\coloneqq TS^{-1}$. Then for every $\varepsilon>0$ satisfying $1>\lambda_2-\varepsilon>0$ and $\lambda_1+\varepsilon \text{Lip}(TS^{-1})<1$, 
\begin{align*}
\|Ru-Rv-(u-v)\|&\leq \lambda_1\|u-v\|+\lambda_2\|Ru-Rv\|\\
&=\lambda_1\|u-v\|+(\lambda_2-\varepsilon)\|Ru-Rv\|+\varepsilon\|Ru-Rv\|\\
&\leq \lambda_1\|u-v\|+(\lambda_2-\varepsilon)\|Ru-Rv\|+\varepsilon \text{Lip}(R) \|u-v\|\\
&= (\lambda_1+\varepsilon \text{Lip}(R))\|u-v\|+(\lambda_2-\varepsilon)\|Ru-Rv\|, \quad \forall u,v \in  \mathcal{Y}.
\end{align*}
 Remaining parts of  the proof is similar to the proof of Theorem \ref{OSECONDTHEOREM}.	
\end{proof}
It is an easy observation that the constant $\lambda_1$ in Theorem \ref{LASTTHEOREM} can not be improved. We are therefore left with the following open problem.
\begin{question}
	Can  the constant $\lambda_2$ be improved further in Theorem  \ref{LASTTHEOREM}?
\end{question}

\section{Applications}
Our first two applications of Theorem \ref{IMPORTANTTHEOREM} are  easy proofs of Soderlind-Campanato perturbation and Barbagallo-Ernst-Thera perturbation.

\begin{theorem}\cite{CAMPANATO, SODERLIND}\label{SODERLINDCAMPANATOAPP} (\textbf{Soderlind-Campanato perturbation})
	Let $ \mathcal{X}$ be a real Banach space,   $ A : \mathcal{X}\rightarrow \mathcal{X}$ be a map  and there exist  $\alpha>0, 0\leq \beta <1$ such that 
	\begin{align*} 
	\|\alpha Ax-\alpha Ay-(x-y)\|\leq\beta\|x-y\|,\quad \forall x,y \in  \mathcal{X}.
	\end{align*} 
	Then $A$ is Lipschitz, invertible and $\text{Lip}(A^{-1})\leq \frac{\alpha}{1-\beta}$.	
\end{theorem}
\begin{proof}
	Set $T=	\alpha A$ and $\lambda_1=\beta$ in Theorem \ref{IMPORTANTTHEOREM}. Then $\frac{1}{\alpha}\text{Lip}(A^{-1})=\text{Lip}(T^{-1})\leq \frac{1}{1-\lambda_1}= \frac{1}{1-\beta}$.	
\end{proof}
\begin{theorem}\cite{BARBAGALLO}\label{BARBAGALLOAPP} (\textbf{Barbagallo-Ernst-Thera perturbation})
	Let $ \mathcal{X}$ be a real Banach space,   $ A : \mathcal{X}\rightarrow \mathcal{X}$ be a map  and there exist  $\alpha>0, 0\leq \beta <1$ such that 
	\begin{align} \label{BET}
	\| Ax- Ay-(\alpha x-\alpha y)\|\leq\beta \|Ax-Ay\|,\quad \forall x,y \in  \mathcal{X}.
	\end{align} 
	Then 
	\begin{enumerate}[\upshape(i)]
		\item If $\beta<1/2, $ then $A$ is Lipschitz, invertible and $\text{Lip}(A^{-1})\leq \frac{1-\beta}{\alpha(1-2\beta)}$.
		\item If $ \mathcal{X}$ is a Hilbert space, then $A$ is Lipschitz, invertible and $\text{Lip}(A^{-1})\leq \frac{1+\beta}{\alpha}$. 
	\end{enumerate}
\end{theorem}
\begin{proof}
	Set $T=	\frac{1}{\alpha} A$ and $\lambda_2=\beta$ in Theorem \ref{IMPORTANTTHEOREM}. Then 
	\begin{enumerate}[\upshape(i)]
		\item   $\alpha\text{Lip}(A^{-1})=1+\lambda_2=1+\beta\leq \frac{1-\beta}{1-2\beta}$.
		\item $\alpha\text{Lip}(A^{-1})=1+\lambda_2=1+\beta$.
	\end{enumerate}
\end{proof}
\begin{remark}
	In \cite{BARBAGALLO}, it is stated that ``the number $1/2$ in Theorem \ref{BARBAGALLOAPP} can not be improved" (see line -7, page 20 , line -14, page 20, and line 14, page 18 in \cite{BARBAGALLO}). In an attempt of giving an example, an operator  $A$ is constructed  which satisfies $\|A-I_\mathcal{X}\|=1/2\|A\|$.   Unfortunately this example does not satisfy Inequality (\ref{BET}).  Reason is the following. The operator $A$ is constructed as follows. Let $\mathcal{X}$ be a Banach space and assume that $\mathcal{X}=\mathcal{Y} \oplus \mathcal{Z}$ for some closed subspaces $\mathcal{Y}$ and $\mathcal{Z}$ of $\mathcal{X}$. Then the projections $P:\mathcal{Y} \oplus \mathcal{Z}\ni y\oplus z \to y \in \mathcal{X}$, $Q:\mathcal{Y} \oplus \mathcal{Z}\ni y\oplus z \to z \in \mathcal{X}$  are bounded linear operators. Define $A\coloneqq 2P$ which is linear. Then
	\begin{align*}
		A-I_\mathcal{X}=2P-(P+Q)=P-Q.
	\end{align*}
	Now for $x=y\oplus z\in \mathcal{Y} \oplus \mathcal{Z}$, we have
	\begin{align*}
		\|Ax-x\|=\|Px-Qx\|=\|y\oplus (-z)\| \quad \text{ and } \quad \|Ax\|=\|2Px\|=2\|y\|.
	\end{align*}
	Thus there is no $\beta\geq 0$ such that 
	\begin{align*}
		\|Ax-x\|=\|y\oplus z\|\leq 2\beta\|y\|=\beta  \|Ax\|,\quad \forall x=y\oplus z\in \mathcal{Y} \oplus \mathcal{Z}.
	\end{align*}   In view of Theorem \ref{IMPORTANTTHEOREM}, we see that the statement given in  \cite{BARBAGALLO}  ``the number $1/2$ in Theorem \ref{BARBAGALLOAPP} can not be improved" is false. However, the statement ``$1/2$ is optimal for symmetry modulus'' is true.
\end{remark}
 We now give applications to the theory of frames.  Paley-Wiener theorem for orthonormal basis in Hilbert spaces inspired the study of perturbation of frames for Hilbert spaces. This was first derived by Christensen in his two  papers \cite{CHRISTENSENPALEY1, CHRISTENSENPALEY2}. This motivated the perturbation of frames and atomic decompositions for Banach spaces  \cite{CHRISTENSENHEIL}. Crucial result used in all these perturbation results is the Neumann series. Later, using Theorem \ref{CASAZZAKALTONTHEOREM}, Casazza and Christensen \cite{CASAZZACHRISTENSEN} improved the results obtained in paper \cite{CHRISTENSENPALEY2}. Using Theorem \ref{CASAZZAKALTONTHEOREM} Stoeva made a systematic study of perturbations of frames for Banach spaces \cite{STOEVA}. For the sake of completeness, we note that Theorem \ref{CASAZZAKALTONTHEOREM} was used in the study of perturbations of frames for Hilbert C*-modules  \cite{HANJING}.\\
  Large body of work on  frames for Hilbert spaces (see \cite{DUFFINSCHAEFFER, CHRISTENSENBOOK, HANLARSONMEMOIRS, HANKORNELSON}) lead to the well developed  theory of  frames (known as Banach frames and $\mathcal{X}_d$-frames) for Banach spaces (see \cite{GROCHENIG,  CASAZZAHANLARSON, CASAZZACHRISTENSENSTOEVA, CHRISTENSENSTOEVAP})  lead to the beginning of frames for frames for metric spaces (known as metric frames) \cite{KRISHNAJOHNSON}. 

For stating these definitions we need the   definition 
of BK-space (Banach scalar valued sequence space or Banach co-ordinate space).
\begin{definition} \cite{BANASMURSALEEN}
	A sequence space $\mathcal{M}_d$ is said to be  a BK-space if all the coordinate functionals are continuous, i.e.,
	whenever $\{x_n\}_n$ is a sequence in $\mathcal{M}_d$ converging to $x \in \mathcal{M}_d$, then each coordinate of 
	$x_n$ converges to each coordinate of $x$.
\end{definition}
\begin{definition}\cite{KRISHNAJOHNSON}\label{METRICBANACHFRAME}
	Let $\mathcal{M}$ be a metric space and $\mathcal{M}_d$ be an associated  BK-space. Let
	$\{f_n\}_{n}$ be a collection in $\operatorname{Lip}(\mathcal{M}, \mathbb{K})$ (where $\mathbb{K}=\mathbb{R}$ or  $\mathbb{C}$)
	and $S: \mathcal{M}_d \rightarrow \mathcal{M}$. If: 
	\begin{enumerate}[\upshape(i)]
		\item $\{f_n(x)\}_{n} \in \mathcal{M}_d$, for each  $x \in \mathcal{M}$,
		\item There exist positive $a, b$  such that 
		$
		a\, d(x,y) \leq \|\{f_n(x)-f_n(y)\}_n\|_{\mathcal{M}_d} \leq b\, d(x,y),  \forall x
		, y\in \mathcal{M},
		$
		\item $S$ is Lipschitz and $S(\{f_n(x)\}_{n})=x$, for each $x \in \mathcal{M}$,
	\end{enumerate}
	then we say that $(\{f_n\}_{n}, S)$ is a \textbf{metric frame}  for $\mathcal{M}$ with respect to  $\mathcal{M}_d$.  We say  constant $a$ as lower frame bound
	and constant $b$ as upper frame bound. 
\end{definition}
As noted in \cite{KRISHNAJOHNSON}, we observe that if $(\{f_n\}_{n}, S)$ is a metric frame  for $\mathcal{M}$ w.r.t. $\mathcal{M}_d$, then  Definition \ref{METRICBANACHFRAME} tells 
that the analysis map
\begin{align*}
\theta_f:\mathcal{M} \ni x \mapsto \theta_f x\coloneqq \{f_n(x)\}_{n} \in \mathcal{M}_d
\end{align*}
is well-defined and  bi-Lipschitz.  Further, $S\theta_f =I_\mathcal{M}$. Now we improve  Theorem 4 in \cite{CASAZZACHRISTENSEN} to metric frames for Banach spaces.

\begin{theorem}\label{STABILITYMA}
	Let $(\{f_n\}_{n}, S)$ be a
	metric frame with lower frame bound $a$ and upper frame bound $b$ for a Banach space    $\mathcal{X}$ w.r.t. $\mathcal{M}_d$. Let $T: \mathcal{M}_d \rightarrow \mathcal{X}$  be a Lipschitz map and suppose that  there exist $\lambda_1, \lambda_2, \mu\geq 0$ such that $\max\{\lambda_2, \lambda_1+\mu b\}<1$  and 
	\begin{align}\label{PERINEQ}
	\|S\{c_n\}_n-S\{d_n\}_n-(T\{c_n\}_n-T\{d_n\}_n)\|&\leq \lambda_1\|S\{c_n\}_n-S\{d_n\}_n\|+\lambda_2\|T\{c_n\}_n-T\{d_n\}_n\|\nonumber\\
	&\quad +\mu \|\{c_n-d_n\}_n\|, \quad \forall \{c_n\}_n,\{d_n\}_n \in  \mathcal{M}_d.
	\end{align}
	Then there exists a collection $\{g_n\}_{n}$ in $\operatorname{Lip}(\mathcal{X}, \mathbb{K})$ such that $(\{g_n\}_{n}, T)$ is  a metric frame for   $\mathcal{X}$ with lower and upper frame bounds
	\begin{align*}
	\frac{a(1-\lambda_2)}{1+\lambda_1+\mu b}, \quad \frac{b(1+\lambda_2)}{1-(\lambda_1+\mu b)}
	\end{align*}
	respectively.
\end{theorem}
\begin{proof}
Given $x, y \in \mathcal{X}$, by taking $\{c_n\}_n$ as $\theta_fx$ and $\{d_n\}_n$ as $\theta_fy$ in Inequality (\ref{PERINEQ}) we get 
\begin{align*}
\|S\theta_fx-S\theta_fy-(T\theta_fx-T\theta_fy)\|\leq \lambda_1 \|S\theta_fx-S\theta_fy\|+\lambda_2 \|T\theta_fx-T\theta_fy\|+\mu\|\theta_fx-\theta_fy\|, \quad \forall x,y \in \mathcal{X}.
\end{align*}
But $S\theta_fx=x, \forall x \in \mathcal{X}$ and hence 
\begin{align*}
\|x-y-(T\theta_fx-T\theta_fy)\|&\leq \lambda_1 \|x-y\|+\lambda_2 \|T\theta_fx-T\theta_fy\|+\mu\|\theta_fx-\theta_fy\|\\
&\leq (\lambda_1+ \mu \operatorname{Lip}(\theta_f))\|x-y\|+\lambda_2 \|T\theta_fx-T\theta_fy\|\\
&\leq (\lambda_1+ \mu b)\|x-y\|+\lambda_2 \|T\theta_fx-T\theta_fy\|, \quad \forall x,y \in \mathcal{X}.
\end{align*}	
Theorem \ref{IMPORTANTTHEOREM} now says that the map $T\theta_f$ is Lipschitz invertible and
\begin{align*}
\frac{1-\lambda_2}{1+\lambda_1+ \mu b}\leq\text{Lip}(T\theta_f)^{-1}\leq\frac{1+\lambda_2}{1-(\lambda_1+ \mu b)}.
\end{align*} 
Define $g_n\coloneqq f_n(T\theta_f)^{-1}$ for all $n \in \mathbb{N}$. Then $g_n$ is Lipschitz for all $n$, $\{g_n(x)\}_{n} \in \mathcal{M}_d$, for each  $x \in \mathcal{M}$ and
\begin{align*}
 \|\{g_n(x)-g_n(y)\}_n\|&=\|\{f_n((T\theta_f)^{-1}(x))-f_n((T\theta_f)^{-1}(y))\}_n\|\leq b \|(T\theta_f)^{-1}(x)-(T\theta_f)^{-1}(y)\|\\
&\leq b \frac{1+\lambda_2}{1-(\lambda_1+ \mu b)}\|x-y\|, \quad \forall x,y \in \mathcal{X},
\end{align*}
\begin{align*}
a\frac{1-\lambda_2}{1+\lambda_1+ \mu b}\|x-y\|&\leq a \|(T\theta_f)^{-1}(x))-(T\theta_f)^{-1}(y)\|\leq \|\{f_n((T\theta_f)^{-1}(x))-f_n((T\theta_f)^{-1}(y))\}_n\|\\
&= \|\{g_n(x)-g_n(y)\}_n\|, \quad \forall x,y \in \mathcal{X}
\end{align*}
which establish lower and upper frame bounds.
Further,  
\begin{align*}
T\{g_n(x)\}_{n}=T\{f_n(T\theta_f)^{-1}(x)\}_{n}=T\theta_f((T\theta_f)^{-1}(x))=x, \quad \forall x \in \mathcal{X}.
\end{align*}
Therefore $(\{g_n\}_{n}, T)$ is  a metric frame for  $\mathcal{X}$.
\end{proof}
We now give another application of Theorem \ref{IMPORTANTTHEOREM}. For this purpose, we introduce  the notion of non-linear atomic decompositions.
\begin{definition}\label{ATOMICMETRIC}
	Let $\mathcal{X}$ be  a Banach  space $\mathcal{X} $ and $\mathcal{M}_d$ be a BK-space. Let
	$\{f_n\}_{n}$ be a sequence in $\operatorname{Lip}(\mathcal{X}, \mathbb{K})$
	and $\{\tau_n\}_{n}$ to be a sequence in   $\mathcal{X}$ If: 
	\begin{enumerate}[\upshape(i)]
		\item $\{f_n(x)\}_{n} \in \mathcal{M}_d$, for each  $x \in \mathcal{X}$,
		\item There exist positive $a, b$  such that 
		\begin{align*}
			a\, \|x-y\| \leq \|\{f_n(x)-f_n(y)\}_n\|_{\mathcal{M}_d} \leq b\, \|x-y\|, \quad \forall x
			, y\in \mathcal{X},
		\end{align*}
		\item $x=\sum_{n=1}^{\infty}f_n(x)\tau_n$, for each $x \in \mathcal{X}$,
	\end{enumerate}
	then we say that $(\{f_n\}_{n}, \{\tau_n\}_{n})$ is a \textbf{Lipschitz  atomic decomposition}   for $\mathcal{X}$ with respect to  $\mathcal{M}_d$. We say  constant $a$ as lower Lipschitz atomic  bound
	and constant $b$ as upper Lipschitz atomic  bound.
\end{definition}
In \cite{CASAZZAHANLARSON} it is proved that not every Banach space admits an atomic decomposition. Motivated from this, we ask the following open problem. 
\begin{question}
	Classify Banach spaces which admit Lipschitz atomic decompositions.
	\end{question}
Following proposition gives various examples of Lipschitz atomic decompositions.
\begin{proposition}\label{LIPATOMICPRO}
 Let $(\{g_n\}_n, \{\omega_n\}_n)$  be an atomic decomposition for a Banach space $\mathcal{Y}$ w.r.t.  BK-space $\mathcal{M}_d$. Let $\mathcal{X}$  be a Banach space and let $A:\mathcal{X}\to \mathcal{Y}$ be a bi-Lipschitz map such that there exists a linear map 	$A:\mathcal{Y}\to \mathcal{X}$ satisfying $BA=I_\mathcal{X}$. Then $(\{f_n\coloneqq g_nA\}_n, \{\tau_n\coloneqq B\omega_n\}_n)$ is a Lipschitz atomic decomposition $\mathcal{X}$ w.r.t. $\mathcal{M}_d$. In particular, if a Banach space admits a Schauder basis, then it admits a Lipschitz atomic decomposition.
\end{proposition}
Particular case of Proposition \ref{LIPATOMICPRO} gives the following example.
\begin{example}
Let $(\{g_n\}_n, \{\omega_n\}_n)$  be an atomic decomposition for a Banach space $\mathcal{X}$ w.r.t. a BK-space $\mathcal{M}_d$. Let 	$T:\mathcal{X}\to \mathcal{X}$ be any bi-Lipschitz map. Define $A:\mathcal{X}\ni x \mapsto (x, Tx) \in  \mathcal{X}\oplus \mathcal{X}$ and $B:\mathcal{X}\oplus \mathcal{X} \ni (x,y)\mapsto x \in \mathcal{X}$. Then $A$ is bi-Lipschitz, $B$  is linear and satisfies $BA=I_\mathcal{X}$. Hence $(\{f_n\coloneqq g_nA\}_n, \{\tau_n\coloneqq B\omega_n\}_n)$ is a Lipschitz atomic decomposition for $\mathcal{X}$ w.r.t. $\mathcal{M}_d$.
\end{example}
At this point it seems  that it is best to develop some theory of Lipschitz atomic decompositions before giving an application of Theorem \ref{IMPORTANTTHEOREM}.

Proposition 2.3 in \cite{CASAZZAHANLARSON} shows that under certain conditions, there is a close relationship between Banach frames and atomic decompositions. Following is the non-linear version of that result. 
\begin{proposition}
	Let $\mathcal{X}$ be  a Banach space  and $\mathcal{M}_d$ be a BK-space. Let
	$\{f_n\}_n$ be a sequence in  $\operatorname{Lip}(\mathcal{X}, \mathbb{K})$ and $S:\mathcal{M}_d\rightarrow
	\mathcal{X}$ be a bounded linear operator. If the standard unit
	vectors  $\{e_n\}_n$ form a Schauder basis for   $\mathcal{M}_d$, then the following are equivalent.
	\begin{enumerate}[\upshape(i)]
		\item $(\{f_n\}_n, S)$ is a metric frame for  $\mathcal{X}$.
		\item $(\{f_n\}_n, \{Se_n\}_n)$ is a Lipschitz  atomic decomposition for  $\mathcal{X}$
		w.r.t.  $\mathcal{M}_d$.
	\end{enumerate}
\end{proposition}
\begin{proof}
	We set $\tau_n=Se_n,\forall n \in \mathbb{N}$ and see that 
	\begin{align*}
	\sum_{n=1}^\infty f_n(x)\tau_n=\sum_{n=1}^\infty f_n(x)Se_n=S\left(\sum_{n=1}^\infty f_n(x)e_n\right)=S\left(\{f_n(x)\}_n\right), \quad \forall x \in 
	\mathcal{X}.
	\end{align*}
	\end{proof}
Well established dilation theory of frames for Hilbert spaces says that frames are images of Riesz bases under projections \cite{CZAJA, KASHINKUKILOVA, HANLARSONMEMOIRS}. This result has been   extended to frames and atomic decompositions for Banach spaces \cite{HANLARSONLIUJOURNALFUN, HANLARSONLIUCON, LARSONSZAFRANIEC, HANLARSONLIULIU, CASAZZAHANLARSON}.  In the next theorem we derive a dilation result for Lipschitz atomic decompositions. We need a proposition  to use in the theorem.
\begin{proposition}\cite{LINDENSTRAUSSBOOK}\label{SBCHAR}
	A sequence 	$\{\tau_n\}_{n}$ in a Banach space $\mathcal{X}$ is a Schauder basis for $\mathcal{X}$ if and only if the following three conditions hold.
	\begin{enumerate}[\upshape(i)]
		\item $ \tau_n\neq 0$ for all $n$.  
		\item There exists $b>0$ such that for every sequence $\{a_k\}_{k}$ of scalars and every pair of natural numbers $n<m$, we have 
		\begin{align*}
		\left\|\sum_{k=1}^na_k\tau _k\right\|\leq b \left\|\sum_{k=1}^ma_k\tau _k\right\|.
		\end{align*}
		\item  $\overline{\operatorname{span}}\{\tau_n\}_{n}=\mathcal{X}$.
		
	\end{enumerate}
\end{proposition}
\begin{theorem}\label{LIPPEL}
	Let $(\{f_n\}_{n}, \{\tau_n\}_{n})$ be a Lipschitz  atomic decomposition for a Banach space $\mathcal{X}$ w.r.t. $\mathcal{M}_d$.
	Then there is a Banach space $\mathcal{Z}$ with a  Schauder basis $\{\omega_n\}_{n}$, an injective map $\theta: \mathcal{X} \to
	\mathcal{Z}$ and a 
 map $P:\mathcal{Z}\rightarrow \mathcal{Z}$ satisfying $P(\mathcal{Z})=\mathcal{X}$, 
	$P^2=P$ and
	$P\omega_n=\theta\tau_n, \forall n \in \mathbb{N}$.
	
\end{theorem}	
\begin{proof}
We generalize the idea of proof of Theorem 2.6 in \cite{CASAZZAHANLARSON} (which is motivated from the arguments in \cite{PELCZYNSKI}) to non-linear setting.	Let $c_{00}$ be the vector space of scalar sequences with only finitely many non-zero terms. Let $\{e_n\}_{n}$ be the standard unit vectors in $c_{00}$. \\
	Case (i): $\tau_n\neq 0$, for all $n$. We define a norm  on $c_{00}$ as follows. Let $\{a_n\}_{n} \in c_{00}$. Define 
	\begin{align}\label{NORMDEF}
	\left\|\sum_{n=1}^\infty a_ne_n\right\|\coloneqq \max _{n}\left\|\sum_{k=1}^na_k\tau_k\right\|.
	\end{align}
	Proposition \ref{SBCHAR} then tells that $\{e_n\}_{n}$ is a Schauder basis for the completion of $c_{00}$, call as $\mathcal{Z}$ w.r.t. just defined  norm. Define 
	\begin{align*}
	\theta: \mathcal{X} \ni x \mapsto \theta x\coloneqq \sum_{n=1}^\infty f_n(x)e_n \in \mathcal{Z}.
	\end{align*}
	From the first condition of the definition of Lipschitz atomic decomposition, from Definition \ref{NORMDEF} and from the construction of $\mathcal{Z}$, it follows that $\theta$ is well-defined.  From the third condition of definition of atomic decomposition, $\theta$ is injective. We next define 
	\begin{align*}
	\Gamma: \mathcal{Z} \ni \sum_{n=1}^\infty a_ne_n \mapsto \Gamma \left(\sum_{n=1}^\infty a_ne_n\right)\coloneqq \sum_{n=1}^\infty a_n\tau_n \in \mathcal{X}.
	\end{align*}
	By verifying $\Gamma$ is bounded linear on dense subset $c_{00}$ of $\mathcal{Z}$, we see that $\Gamma$ is bounded on $\mathcal{Z}$. Then 
	\begin{align}\label{PISONTO}
	\Gamma \theta x=\Gamma\left(\sum_{n=1}^\infty f_n(x)e_n\right)=\sum_{n=1}^\infty f_n(x)\tau_n=x, \quad \forall x \in \mathcal{X}.
	\end{align}
	So if we define $P\coloneqq \theta \Gamma$, then $P^2=\theta \Gamma\theta \Gamma=\theta \Gamma=P$. Equation (\ref{PISONTO}) tells that $P(\mathcal{Z})=\mathcal{X}$. We next see that $Pe_n=\theta \Gamma e_n=\theta \tau_n$, $\forall n $. Thus we can take $\omega_n=e_n$, for all $n$ to get the result. \\
	Case (ii): $\tau_n= 0$, for some $n$. Let $J=\{n:\tau_n\neq 0\}$. We now apply case (i) to the collection $\{a_n\}_{n\in J}$. Let $\theta$, $\mathcal{Z}$, $\Gamma$ and $P$ be as in the case (i). Without affecting the definition of atomic decomposition, we can take  $f_n= 0$ for all $n \in J^c$. Now consider the space $\mathcal{Z}\oplus \ell^2(J^c)$ and let $\{\rho_n\}_{n\in J^c}$ be an orthonormal basis for $\ell^2(J^c)$. Define $Q:\mathcal{Z}\oplus \ell^2(J^c) \ni z\oplus y\mapsto Q(z\oplus y)\coloneqq P z\oplus 0 \in \mathcal{Z}\oplus \ell^2(J^c)$. Now the space $\mathcal{Z}\oplus \ell^2(J^c)$ has  Schauder basis $\{\tau_n \oplus 0, 0\oplus \rho_m\}_{n \in J, m\in J^c}$ and  $Q$ satisfy the conclusions. Thus we can take $\omega_n=e_n$, for all $n\in J$ and  $\omega_n=\rho_n$, for all $n\in J^c$ to get the result.
\end{proof}
We leave the further study of Lipschitz atomic decomposition to future work and end the paper with an application of Theorem \ref{IMPORTANTTHEOREM}.
\begin{theorem}\label{APPLICATIONTHEOREMLAST}
Let $(\{f_n\}_{n}, \{\tau_n\}_{n})$ be  a Lipschitz  atomic decomposition with  lower  Lipschitz atomic bound $a$ and upper Lipschitz atomic  bound $b$  for $\mathcal{X}$ w.r.t. $\mathcal{M}_d$. 	Let $\{\omega_n\}_{n}$ be a collection in $ \mathcal{X}$ and suppose that  there exist $\lambda_1, \lambda_2, \mu\geq 0$ such that the following conditions hold.
\begin{enumerate}[\upshape(i)]
	\item For each $x\in  \mathcal{X}$, the series $\sum_{n=1}^{\infty}f_n(x)\omega_n$ converges in $\mathcal{X}$.
	\item $\max\{\lambda_2, \lambda_1+\mu b\}<1$  and 
	\item 
\end{enumerate}  
\begin{align}\label{PERINEQLIP}
\left\|\sum_{n=1}^{\infty}(c_n-d_n)(\tau_n-\omega_n)\right\|&\leq \lambda_1\left \|\sum_{n=1}^{\infty}(c_n-d_n)\tau_n\right\|+\lambda_2\left \|\sum_{n=1}^{\infty}(c_n-d_n)\omega_n\right\|\nonumber\\
&\quad +\mu \|\{c_n-d_n\}_n\|, \quad \forall \{c_n\}_n,\{d_n\}_n \in  \mathcal{M}_d.
\end{align}
Then there exists a collection  $\{g_n\}_{n}$  in $\operatorname{Lip}(\mathcal{X}, \mathbb{K})$ such that $(\{g_n\}_{n}, \{\omega_n\}_{n})$ is   a Lipschitz  atomic decomposition for $\mathcal{X}$ with lower and upper Lipschitz atomic bounds 
\begin{align*}
\frac{a(1-\lambda_2)}{1+\lambda_1+\mu b}, \quad \frac{b(1+\lambda_2)}{1-(\lambda_1+\mu b)}
\end{align*}
respectively.
\end{theorem}
\begin{proof}
From the first condition we get that the map $T: \mathcal{X} \ni x \mapsto \sum_{n=1}^{\infty}f_n(x)\omega_n \in \mathcal{X}$ is well-defined. Now using  Inequality 	(\ref{PERINEQLIP}),
\begin{align*}
\|x-y-(Tx-Ty)\|&=\left\| \sum_{n=1}^{\infty}(f_n(x)-f_n(y))\tau_n- \sum_{n=1}^{\infty}(f_n(x)-f_n(y))\omega_n\right\|\\
&=\left\|\sum_{n=1}^{\infty}(f_n(x)-f_n(y))(\tau_n-\omega_n)\right\|\\
&\leq \lambda_1 \left\|\sum_{n=1}^{\infty}(f_n(x)-f_n(y))\tau_n\right\|+\lambda_2\left\|\sum_{n=1}^{\infty}(f_n(x)-f_n(y))\omega_n\right\| +\mu \|\{f_n(x)-f_n(y)\}_n\|\\
&=\lambda_1 \left\|x-y\right\|+\lambda_2\left\|Tx-Ty\right\| +\mu \|\{f_n(x)-f_n(y)\}_n\|\\
&\leq (\lambda_1+\mu b)\left\|x-y\right\|+\lambda_2\left\|Tx-Ty\right\|, \quad \forall x, y \in \mathcal{X}.
\end{align*}
 Theorem \ref{IMPORTANTTHEOREM} then says that $T$ is Lipschitz, invertible and 
 \begin{align*}
 \frac{1-\lambda_2}{1+\lambda_1+ \mu b}\leq\text{Lip}(T)^{-1}\leq\frac{1+\lambda_2}{1-(\lambda_1+ \mu b)}.
 \end{align*}
 Define $g_n\coloneqq f_nT^{-1}$, $\forall n \in \mathbb{N}$. Then $\{g_n(x)\}_{n} \in \mathcal{M}_d$, for each  $x \in \mathcal{X}$ and
 \begin{align*}
 \|\{g_n(x)-g_n(y)\}_n\|&= \|\{f_n(T^{-1}x)-f_n(T^{-1}y)\}_n\|_n\leq b\|T^{-1}x-T^{-1}y\|\\
 &\leq b \frac{1+\lambda_2}{1-(\lambda_1+ \mu b)}\|x-y\|, \quad \forall x,y \in \mathcal{X},
 \end{align*}
 \begin{align*}
 a\frac{1-\lambda_2}{1+\lambda_1+ \mu b}\|x-y\|&\leq a \|T^{-1}x-T^{-1}y\|\leq \|\{f_n(T^{-1}x)-f_n(T^{-1}y)\}_n\|\\
 &= \|\{g_n(x)-g_n(y)\}_n\|, \quad \forall x,y \in \mathcal{X}.
 \end{align*}
 Finally 
 \begin{align*}
 \sum_{n=1}^{\infty}g_n(x)\omega_n=\sum_{n=1}^{\infty}f_n(T^{-1}x)\omega_n=T(T^{-1}x)=x, \quad \forall x \in \mathcal{X}.
 \end{align*}
\end{proof}
We conclude the paper with the following remarks.
\begin{remark}
We can improve Theorems \ref{SODERLINDCAMPANATOAPP},  \ref{BARBAGALLOAPP},   \ref{STABILITYMA} and 	 \ref{APPLICATIONTHEOREMLAST}  using Theorem \ref{LASTTHEOREM}.
\end{remark}
\begin{remark}
So far in the literature, 	there are three ways to prove Theorem \ref{CASAZZAKALTONTHEOREM} one given in \cite{CASAZZAKALTON}, another in \cite{VANEIJNDHOVEN} and yet another in \cite{CASAZZACHRISTENSEN}. As we mentioned earlier, we have done the non-linear version of arguments used in \cite{VANEIJNDHOVEN}. We hope that arguments used in \cite{CASAZZAKALTON}  and \cite{CASAZZACHRISTENSEN} can be generalized to give different proofs of Theorem \ref{IMPORTANTTHEOREM}.
\end{remark}
\section{Acknowledgements}
I thank Dr. P. Sam Johnson, Department of Mathematical and Computational Sciences, National Institute of Technology Karnataka (NITK),  Surathkal for several discussions.


 \bibliographystyle{plain}
 \bibliography{reference.bib}

\end{document}